\documentclass[12pt,reqno]{amsart}
\usepackage{amsmath,amssymb,amsthm}
\usepackage{verbatim}
\usepackage[english]{babel}
\usepackage[utf8]{inputenc}
\usepackage{tikz-cd}
\usepackage{hyperref}
\usepackage{enumerate}
\usepackage{cleveref}
\usepackage{multirow,float}
\usepackage{csquotes}

\usepackage[sorting=nyc,doi=false,isbn=false,url=false,giveninits=true,maxbibnames=99]{biblatex}

\DeclareSortingScheme{nyc}{
  \sort{
    \field{presort}
  }
  \sort[final]{
    \field{sortkey}
  }
  \sort{
    \field{sortname}
    \field{author}
    \field{editor}
    \field{translator}
    \field{sorttitle}
    \field{title}
  }
  \sort{
    \field{sortyear}
    \field{year}
  }
  \sort{\citeorder}
}

\AtEveryBibitem{%
\ifentrytype{article}{
    \clearfield{number}
}{}
}
\renewbibmacro{in:}{%
  \ifentrytype{article}{}{\printtext{\bibstring{in}\intitlepunct}}}
\AtEveryBibitem{%
\ifentrytype{book}{
    \clearfield{pages}
}{}
}

\tikzstyle{nodal}=[circle,draw,fill=black,inner sep=0pt, minimum width=4pt]

\tikzset{double distance = 2pt}

\newtheorem{theorem}{Theorem}[section]
\newtheorem{conjecture}[theorem]{Conjecture}

\newtheorem{lemma}[theorem]{Lemma}
\newtheorem{corollary}[theorem]{Corollary}
\newtheorem{proposition}[theorem]{Proposition}

\newtheorem{theoremvar}{Theorem}
\newenvironment{theorem*}[1]
 {\renewcommand{\thetheoremvar}{#1}\theoremvar}
 {\endtheoremvar}
\newtheorem{corollaryvar}{Corollary}
\newenvironment{corollary*}[1]
 {\renewcommand{\thecorollaryvar}{#1}\corollaryvar}
 {\endcorollaryvar}

\theoremstyle{definition}
\newtheorem{definition}[theorem]{Definition}
\newtheorem{example}[theorem]{Example}
\newtheorem{remark}[theorem]{Remark}

\usepackage{geometry}
\newcommand{\BC}{{\mathbb {C}}}

\newcommand{\BP}{{\mathbb {P}}}
\newcommand{\BQ}{{\mathbb {Q}}}
\newcommand{\QQ}{{\mathbb {Q}}}

\newcommand{\BZ}{{\mathbb {Z}}}
\newcommand{\ZZ}{{\mathbb{Z}}}
\newcommand{\Z}{{\mathbb{Z}}}

\newcommand{\CO}{{\mathcal {O}}}

\newcommand{\CS}{{\mathcal {S}}}

\newcommand{\RH}{{\mathrm {H}}}

\DeclareMathOperator{\Aut}{Aut}

\DeclareMathOperator{\Gal}{Gal}
\DeclareMathOperator{\GL}{GL}
\DeclareMathOperator{\Pic}{Pic}
\DeclareMathOperator{\Num}{Num}
\DeclareMathOperator{\rk}{rk}

\newcommand{\ov}{\overline}

\newcommand{\piet}{\pi_1^{\mathrm{\acute{e}t}}}
\newcommand{\pitop}{\pi_1^{\mathrm{top}}}

\usepackage{color}

\title{A Hilbert Irreducibility Theorem for Enriques surfaces}

\author{Damián Gvirtz-Chen}
\address{Department of Mathematics, University College London, 25~Gordon Street, London, WC1H  0AY, United Kingdom.}
\email{d.gvirtz@ucl.ac.uk}

\author{Giacomo Mezzedimi}
\address{Institut für Algebraische Geometrie, Leibniz Universität Hannover, Wel\-fen\-gar\-ten~1, 30167 Hannover, Germany.}
\email{mezzedimi@math.uni-hannover.de}

\begin{document}

\begin{abstract}
    We define the over-exceptional lattice of a minimal algebraic surface of Kodaira dimension $0$. Bounding the rank of this object, we  prove that a conjecture by Campana and Corvaja--Zannier holds for Enriques surfaces, as well as K3 surfaces of Picard rank $\geq 6$ apart from a finite list of geometric Picard lattices.
    
    Concretely, we prove that such surfaces over finitely generated fields of characteristic $0$ satisfy the weak Hilbert property after a finite field extension of the base field. The degree of the field extension can be uniformly bounded.
\end{abstract}

\maketitle

\section{The conjecture by Campana and Corvaja--Zannier}
The subject of this article is the potential behaviour of rational points on Enriques and K3 surfaces, as contrasted with the behaviour of rational points on their finite covers. After Serre, one can ask whether these are related to each other through the \emph{Hilbert Property}, which loosely says that many rational points do not come from finite covers (see \cite[Definition~1.1]{CDJLZ} for the particular definition given below).

\begin{definition}
 An integral, quasi-projective variety $X$ over a field $k$ has the \emph{Hilbert Property} if for any finite family of finite covers (i.e.\ finite, surjective morphisms) $\phi_i\colon Y_i\to X$, $i=1,\dots,r$, by integral, normal varieties $Y_i$ over $k$ with $\deg(\phi_i)>1$, the set $X(k)\setminus\bigcup_{i=1}^r \phi_i(Y_i(k))$ is Zariski-dense in $X$.
 
 If $X_{k'}$ has the Hilbert Property for some finite field extension $k'/k$, then $X$ has the \emph{potential Hilbert Property}.
\end{definition}

The definition of the potential Hilbert Property is well-motivated by the fact that the Hilbert Property is stable under finite extensions of the base field, as proved in \cite[Proposition~3.2.1]{Serre}.
Following \cite[Theorem~1.6]{CorvajaZannier}, the Chevalley--Weil theorem implies that the Hilbert Property fails as soon as $X$ admits unramified finite covers of degree $>1$, i.e.\ $X$ is not \emph{algebraically simply connected}. It is therefore reasonable to make an alternative definition that behaves better in this situation.
\begin{definition}[{Corvaja--Zannier \cite[\S2.2]{CorvajaZannier}}]
  A smooth proper variety $X$ over a field $k$ has the \emph{weak Hilbert Property} if for any finite family of \textbf{ramified} finite covers $\phi_i\colon Y_i\to X$, $i=1,\dots,r$, by integral, normal varieties $Y_i$ over $k$ with $\deg(\phi_i)>1$, the set $X(k)\setminus\bigcup_{i=1}^r \phi_i(Y_i(k))$ is Zariski-dense in $X$.
 
 If $X_{k'}$ has the weak Hilbert Property for some finite field extension $k'/k$, then $X$ has the \emph{potential weak Hilbert Property}.
\end{definition}

A guiding philosophy is provided by Campana's notion of specialness: a smooth algebraic variety is \emph{special} if it does not admit a fibration of general type in the sense of \cite[Definition~2.41]{Campana}. It follows from \cite[Proposition~9.29]{Campana} that for curves and surfaces, this is equivalent to there being no finite \'etale cover which dominates a positive-dimensional variety of general type. All algebraic varieties of Kodaira dimension $0$ are special by \cite[Theorem~5.1]{Campana}.

The general expectation is that specialness not only controls the rational points on a variety but also the distribution of rational points on its covers, as expressed by the weak Hilbert Property.

\begin{conjecture}[Campana, Corvaja--Zannier]\label{ccz}
 Let $X$ be a smooth, projective, geometrically integral variety over a finitely generated field $K$ of characteristic $0$ (e.g.\ a number field). Then the following are equivalent:
 \begin{enumerate}
  \item $X$ is special.
  \item $X$ satisfies potential density, i.e.\ there exists a finite field extension $K'/K$ such that $X(K')$ is Zariski-dense in $X$.
  \item $X$ satisfies the potential weak Hilbert Property.
 \end{enumerate}
\end{conjecture}

For the remainder of this article, let $K$ thus be a finitely generated field of characteristic $0$, and $\ov K$ an algebraic closure of $K$. Minimal algebraic surfaces over $K$ of Kodaira dimension $0$ occur in two classes.

The first class comprises geometrically abelian surfaces, as well as their bi-elliptic quotients. It has been long known, e.g.\ from \cite[Theorem~10.1]{FreyJarden}, that potential density holds for these surfaces. The recent work \cite{CDJLZ} by Corvaja, Demeio, Javanpeykar, Lombardo and Zannier established the weak Hilbert Property for abelian varieties with Zariski-dense set of rational points, thus settling \Cref{ccz} in this case.

The second class comprises K3 surfaces and Enriques surfaces. The latter arise as étale quotients of K3 surfaces by an involution. While it is still not known whether K3 and Enriques surfaces can exhibit a non-empty, not Zariski-dense set of rational points, potential density of rational points has been proved by Bogomolov and Tschinkel for all Enriques surfaces in \cite{BogomolovTschinkel}, and for K3 surfaces with an elliptic fibration or infinite automorphism group in \cite{BogomolovTschinkel_2}.

\subsection*{Main results}

We settle \Cref{ccz} for all Enriques and many K3 surfaces. 
To our knowledge, these are the first general results on the distribution of rational points on covers of Enriques and K3 surfaces apart from sporadic examples (e.g.\ \cite[Theorem~1.4]{CorvajaZannier}, \cite[\S5]{Demeio_2} and \cite[Propositions~4.3 and 4.4]{Demeio}).

\begin{theorem*}{A}\label{thm:enriques}
 Every Enriques surface $S$ over a finitely generated field $K$ of characteristic $0$ has the potential weak Hilbert Property.
 
More precisely, if $S(K)$ is Zariski-dense in $S$ and $\Pic(S)=\Pic(S_{\ov K})$, then $S$ has the weak Hilbert Property.
\end{theorem*}

Keum proved in \cite{Keum} that every Kummer surface over an algebraically closed field of characteristic $0$ is the \'etale double cover of an Enriques surface. Since the potential weak Hilbert Property is stable under \'etale covers by \cite[Theorem~3.16]{CDJLZ}, one immediately deduces:

\begin{corollary*}{A}\label{cor:kummer}
 Every geometrically Kummer surface $X$ over a finitely generated field $K$ of characteristic $0$ has the potential Hilbert Property.
\end{corollary*}
Curiously, this way of deriving the potential Hilbert Property for Kummer surfaces makes no use of the weak Hilbert Property for abelian surfaces from \cite{CDJLZ}, and in fact $X$ could be taken to be any K3 surface covering an Enriques surface (cf. \Cref{thm:enriques_cover}). We refer the reader to \cite{BSV} and \cite{Sertoz} for characterisations of such $X$. For example, the only singular K3 surfaces (i.e.\ of maximal Picard rank $\rho(X_{\ov K})=\mathrm{rk}\,\mathrm{Pic}(X_{\ov K})=20$) not covering an Enriques surface are those of discriminant $d\equiv -3\bmod 8$, $d=-4,-8$, or $d=-16$ where $X$ is not Kummer.

An immediate question arising from \Cref{cor:kummer} is how far it can be extended to other K3 surfaces. Comparing with the restrictive assumptions in the work of Bogomolov and Tschinkel on potential density, it is not at all surprising that K3 surfaces pose a significantly harder challenge than Enriques surfaces. Nevertheless, a partial answer is given by the following theorem.
\begin{theorem*}{B}\label{thm:K3}
 Let $K$ be a finitely generated field of characteristic $0$.
 
 All K3 surfaces $X$ over $K$ with
 \begin{enumerate}
     \item $\rho(X_{\ov K})<10$ and at least two elliptic fibrations, or
     \item $\rho(X_{\ov K})\geq 6$, except for a finite list $\CS$ of excluded (isomorphism classes of) geometric Picard lattices $\Pic(X_{\ov K})$,
 \end{enumerate}
    have the potential Hilbert Property.
    
    More precisely, if $X(K)$ is Zariski-dense in $X$ and $\Pic(X)=\Pic(X_{\ov K})$, then $X$ has the Hilbert Property.
\end{theorem*}

Notice that since K3 surfaces of Picard rank at least $5$ are geometrically elliptic (see \cite[Proposition~11.1.3(ii)]{Huybrechts}), every K3 surface in \Cref{thm:K3} satisfies potential density of rational points by \cite[Theorem~1.1]{BogomolovTschinkel_2}. In fact, we will see that every K3 surface in \Cref{thm:K3} admits at least one additional elliptic fibration (see \Cref{rk:finitelist}(2) and the proof of \Cref{thm:K3} in Section~\ref{sec:demeio}).

The finite list $\CS$ of excluded geometric Picard lattices is independent of $K$ and not known explicitly (see \Cref{rk:finitelist}). It is likely that a more sophisticated, possibly computer-assisted analysis could narrow down $\CS$.

Regarding \Cref{ccz}, one general challenge is to determine the field extension $K'/K$, over which the (weak) Hilbert Property is acquired. Note that the conjecture does not predict that the extensions over which $X$ acquires potential density or the weak Hilbert Property are the same.\footnote{Though it would follow if the ``more optimistic'' \cite[Question-Conjecture 1]{CorvajaZannier} holds.} Nevertheless, by \cite{CDJLZ} this is true for principal homogeneous spaces of abelian varieties, and we suspect that it is also true for Enriques and K3 surfaces. A related question is whether the degree of $K'$ over $K$ is uniformly bounded by a constant not depending on the surface or base field. For potential density of elliptic K3 surfaces, this has recently been established by Lai and Nakahara in \cite{LaiNakahara}. For the weak Hilbert Property of those surfaces covered by our results, we can build on their work to answer this question too in the affirmative in \Cref{rem:uniform}.

\subsection*{Strategy and outline}
We will finish this introduction by giving an outline of the article and the involved techniques. In Sections~\ref{sec:lattices}, \ref{sec:enriques} and \ref{sec:moduli} we recall and develop facts about abstract lattices, Enriques and K3 surfaces, and root lattices of divisors on K3 surfaces.

The key ingredient in our proof is the \emph{over-exceptional lattice} of a minimal algebraic surface of Kodaira dimension $0$ which we introduce in \Cref{sec:exceptional} as the span of all irreducible curves which are orthogonal to every elliptic fibration inside the Picard lattice. The terminology is explained by an earlier construction of the so-called \emph{exceptional lattice} due to Nikulin in \cite{Nikulin}. While the over-exceptional and exceptional lattice are expected to be of relatively small rank for most surfaces (see \cite{Yu}), so far no general bounds are known in the literature. As a first step in this direction, we prove the following theorem, which we hope will be of independent interest outside the study of the Hilbert Property.

\begin{theorem*}{C}\label{thm:overexceptional_Enriques}
Let $S$ be an Enriques surface over an algebraically closed field of characteristic different from $2$. The over-exceptional lattice $E'(S)$ of $S$ satisfies the following properties:
\begin{enumerate}
    \item $\rk{E'(S)}\le 5$;
    \item The embedding $E'(S)\hookrightarrow \Num(S)$ is primitive;
    \item If $\pi:X\to S$ is the K3 cover of $S$, the embedding $\pi^*E'(S)=E'(S)^{\oplus 2}\hookrightarrow \Pic(X)$ given by the pullback is primitive.
\end{enumerate}
\end{theorem*}

We apply this bound to a lattice-theoretic criterion for the weak Hilbert Property which we derive in Section~\ref{sec:demeio}, building on an idea for establishing the Hilbert Property first used by Corvaja and Zannier \cite[Theorem~1.6]{CorvajaZannier} and later by Demeio \cite{Demeio}. The method involves the construction of a divisor $Z\subset X$ such that $X$ has the Hilbert Property if $X\setminus Z$ is algebraically simply connected and $X(K)$ is Zariski-dense. Combining this with a lattice-theoretic criterion for the simply connectedness of open K3 and Enriques surfaces due to Shimada and Zhang, we are able to derive \Cref{thm:enriques} from the bounds in \Cref{thm:overexceptional_Enriques}.

Motivated by a study of extremal cases and \Cref{thm:overexceptional_Enriques} (see also \Cref{cor:overexceptional_K3}), one may ask whether the following holds.

\begin{conjecture} \label{conj:overexceptional_K3}
 Let $X$ be a K3 surface over an algebraically closed field. Assume that $X$ has infinite automorphism group. Then the embedding of the over-exceptional lattice $E'(X)\hookrightarrow \Pic(X)$ is primitive.
\end{conjecture}

Note that the assumption on the automorphism group is necessary by the counterexample presented in \Cref{rk:finitelist}(1).

Using the ideas from the present article, proving \Cref{conj:overexceptional_K3} would be a crucial step towards the conjecture of Campana and Corvaja--Zannier for K3 surfaces. We hope to revisit this question in future research.

\subsection*{Acknowledgements} We thank Julian Demeio, Matthias Sch\"utt and Dome\-nico Valloni for helpful and stimulating discussions. We also thank the anonymous referees for carefully reading the paper and suggesting several improvements both to content and structure.

\section{Generalities about lattices} \label{sec:lattices}

In this section we recall the basics of lattice theory. Reference works are \cite{Nikulin_3} and \cite{ebeling}.
\begin{definition}
A \emph{lattice} is a finitely generated free $\ZZ$-module $L$ endowed with a non-degenerate symmetric bilinear form with values in $\ZZ$. The \emph{rank} $\rk L$ is the rank of $L$ as a $\ZZ$-module. $L$ is called \emph{even} if $x.x\in 2\ZZ$ for all $x\in L$. The \emph{signature} of $L$ is the signature of the natural extension of the bilinear form to the real vector space $L\otimes \mathbb{R}$. Hence, $L$ is called \emph{positive} (resp.\ \emph{negative}) \emph{definite} if the signature of $L$ is $(\rk L,0)$ (resp.\ $(0,\rk L)$). The \emph{discriminant} $\mathrm{disc}(L)$ of $L$ is the absolute value of the determinant of any matrix representing the bilinear product on $L$.
\end{definition}
\begin{example}
The even lattice $\mathbf{U}$ of rank $2$ with intersection matrix
$$\begin{pmatrix}
0 &1\\
1 &0
\end{pmatrix}$$
is called the \emph{hyperbolic plane}. It has signature $(1,1)$ and it is \emph{unimodular}, that is, $\mathrm{disc}(\mathbf{U})=1$.
\end{example}

We will be mainly interested in even, negative definite lattices for applications to the theory of Enriques and K3 surfaces. Therefore in the following we assume $L$ to be even and negative definite.

To the lattice $L$ we can attach the \emph{dual lattice} 
$$L^\vee = \{x \in L\otimes \QQ \mid x.L\subseteq \ZZ\}.$$
The bilinear product on $L$ naturally extends to $L^\vee$. Moreover $L^\vee$ has the same rank as $L$, so the quotient $A_L = L^\vee /L$ is a finite group. We can endow $A_L$ with the quadratic form $q_L$ with values in $\QQ/2\ZZ$ such that
$$q_L(\overline{v},\overline{v})=v.v \pmod{2\ZZ},$$
where $\overline{v}$ denotes the class of $v\in L^\vee$ in the quotient $A_L$.

\begin{definition}
The finite group $A_L$, endowed with the quadratic form $q_L$, is called the \emph{discriminant group} of $L$.
\end{definition}

Notice that the cardinality of $A_L$ coincides with the discriminant of $L$. The \emph{length} of the group $A_L$, denoted by $\ell(A_L)$, is the minimum number of generators of $A_L$ as an abelian group.

\begin{definition}
An \emph{embedding} of lattices $i:L\hookrightarrow M$ is an injective homomorphism that respects the bilinear pairings. The embedding is called \emph{primitive} if the cokernel $M/i(L)$ is torsion-free. The \emph{saturation} of $L$ in $M$, denoted by $L_{sat}$, is the smallest primitive sublattice of $M$ containing $i(L)$.
\end{definition}

If $L\hookrightarrow M$ is an embedding of lattices, the saturation $L_{sat}$ has the same rank as $L$ and it contains a copy of $L$. It is called an \emph{overlattice} of $L$:

\begin{definition}
An \emph{overlattice} of $L$ is an even lattice $L'$ containing $L$ such that the quotient $L'/L$ is finite. If $L'$ is an overlattice of $L$, we denote by $[L':L]$ the \emph{index} of the overlattice, i.e.\ the index of $L$ as a subgroup of $L'$.
\end{definition}

It is very easy to characterize the overlattices of a given lattice. For instance, if $L'$ is an overlattice of $L$, then the quotient $H=L'/L$ is a subgroup of $A_L$, and moreover the quadratic form $q_L$ is identically zero on $H$. Subgroups of $A_L$ on which $q_L$ is identically zero are called \emph{isotropic}. Conversely, if $H$ is any isotropic subgroup of $A_L$, we can construct the overlattice $L'\subseteq L^\vee$ obtained by adjoining to $L$ the preimages in $L^\vee$ of the generators of $H$. This proves the following:

\begin{proposition}[\cite{Nikulin_3}, Proposition 1.4.1]
There exists a 1:1 correspondence between overlattices of $L$ and isotropic subgroups of $A_L$.
\end{proposition}

If $L'$ is an overlattice of $L$, we can compute the discriminant of $L'$ by using the formula \cite[\S1.4]{Nikulin_3}
\begin{equation} \label{eq:over}
    \frac{\mathrm{disc}(L)}{\mathrm{disc}(L')}=[L':L]^2.
\end{equation}
We conclude the section by introducing the root lattices, the ones we will be mainly interested in.

\begin{definition}
A \emph{root lattice} is an even, negative definite lattice $L$ that admits a set of generators given by vectors of norm $-2$.
\end{definition}

Vectors of norm $-2$ are usually called \emph{roots}. Any root lattice can be decomposed as the direct sum of some ADE lattices \cite[Theorem 1.2]{ebeling}. These are the negative definite lattices $\mathbf{A}_n$ (for $n\ge 1$), $\mathbf{D}_n$ (for $n\ge 4$) and $\mathbf{E}_n$ (for $6 \le n\le 8$) corresponding to simply laced Dynkin diagrams.

We list the discriminant groups of ADE lattices in the following table for convenience:

\begin{table}[H]
  \renewcommand*{\arraystretch}{1.4}
   \centering
\begin{tabular}{ c|c|c } 
$L$ & $A_L$ & $q_L$\\[0.5em]
 \hline\hline
$\mathbf{A}_n$ &$\Z/(n+1)$ &$-\frac{n}{n+1}$\\
\hline
$\mathbf{D}_{2n}$ &$\Z/2 \times \Z/2$ &$\begin{pmatrix} 1 &\frac{1}{2}\\ \frac{1}{2} &-\frac{n}{2} \end{pmatrix}$\\
\hline
$\mathbf{D}_{2n+1}$ &$\Z/4$ &$-\frac{2n+1}{4}$\\
\hline
$\mathbf{E}_6$  &$\Z/3$ &$\frac{2}{3}$\\
\hline
$\mathbf{E}_7$  &$\Z/2$ &$\frac{1}{2}$\\
\hline
$\mathbf{E}_8$  &$\{1\}$ &$0$\\

\end{tabular}
\caption{Discriminant groups of ADE lattices}
\label{ADE}
\end{table}

\section{Generalities about K3 and Enriques surfaces}\label{sec:enriques}
This section recalls useful facts about K3, Kummer and Enriques surfaces. Reference works are \cite{BarthHulek, EnriquesI, EnriquesII}. Throughout the section, $k$ is a field of characteristic different from $2$ and $\ov k$ an algebraic closure of $k$. 

\begin{definition}
A \emph{K3 surface} is a smooth projective algebraic surface $X$ over $k$ such that its canonical bundle $K_X$ is trivial and $\RH^1(X,\CO_X)=0$.
 
An \emph{Enriques surface} is a smooth projective algebraic surface $S$ over $k$ such that $2K_S=0$ and $\RH^1(S,\CO_S)=\RH^2(S,\CO_S)=0$.
\end{definition}

We note that by flat base change a smooth projective algebraic surface over $k$ is a K3 or Enriques surface over $k$ if and only if it so over $\ov k$. In the following, $X$ will be a K3 surface over $k$ and $S$ an Enriques surface over $k$. We will write $X_{\ov k}$ and $S_{\ov k}$ for their base change to the algebraic closure $\ov k$.

The fundamental group $\piet(X_{\ov k})$ is trivial, while the fundamental group of $S_{\ov k}$ is isomorphic to $\BZ/2$. More precisely, there exists a K3 surface $Y$ over $k$ and a fixed-point free involution $\iota$ on $Y$, called an \emph{Enriques involution}, such that $S$ is the quotient of $Y$ by $\iota$. 

The Picard lattice $\Pic(X)$ of $X$ is isomorphic to the group of divisors up to numeric equivalence $\Num(X)$ and has signature $(1,\rho(X)-1)$ where $1\leq\rho(X)\leq 22$, and $\rho(X)\leq 20$ in characteristic $0$.

The Picard lattice $\Pic(S)$ of $S$ has non-trivial torsion coming from $K_S$. Over $\ov k$, its torsion-free part $\Num(S_{\ov k})$ is isomorphic to $\mathbf{U}\oplus \mathbf{E}_8$.

By Riemann-Roch a divisor $D$ representing an integral curve on $X$ or $S$ satisfies $D^2\geq -2$, and $D^2=-2$ if and only if $D$ is a smooth curve of genus $0$, called a \emph{$(-2)$-curve}. 

In the rest of the section we will recall some well-known facts about elliptic fibrations on Enriques and K3 surfaces, as well as some important results about \emph{isotropic sequences}, which will constitute the main technical tool of the paper.

\begin{definition}
 An \emph{elliptic fibration} is a faithfully flat morphism of geometrically integral, smooth algebraic varieties whose generic fibre is a smooth, proper, geometrically integral curve of genus $1$. (In particular, a section is not required.)
\end{definition}

Any elliptic fibration on an Enriques or K3 surface has to be fibred over the projective line.

If $E$ is a primitive, effective divisor on a K3 surface $X$ with $E^2=0$ and its dual graph is one of the extended simply laced Dynkin diagrams with the correct multiplicities (see e.g. \cite[Table~I.4.3]{Miranda}), then $|E|$ yields an elliptic fibration on $X$ (see \cite[Theorem~2.3.10]{Huybrechts}, noticing that such divisors are nef, since they intersect nonnegatively every irreducible curve on $X$). 

If $F$ is an effective divisor on an Enriques surface $S$ with $F^2=0$ and its dual graph is one of the extended simply laced Dynkin diagrams with the correct multiplicities, then either $|F|$ or $|2F|$ yields an elliptic fibration $\pi:S\to \BP^1_k$, depending on whether $F$ is primitive in $\Num(S)$ or not (see \cite[Lemma~(17.3)]{BarthHulek}). There are exactly two multiple fibres of $\pi$ which are of Kodaira types $I_m$ and $I_n$, $m,n\geq 0$, and have multiplicity $2$. The fibres of Kodaira type $I_n$ are called \emph{multiplicative}, while the other ones are \emph{additive}.

Every elliptic fibration on the Enriques surface $S$ arises as a pencil $|2F|$ with the aforementioned properties. The curve $F$ is called a \emph{half-fibre}, while the non-multiple fibres in $|2F|$ are called \emph{simple} fibres. In particular additive fibres are always simple.

We have the following bound on the number of vertical components in reducible fibres of $|2F|$:

\begin{proposition} \label{prop:embedE8}
Let $S$ be an Enriques surface and $|2F|$ an elliptic fibration on $S$. The number of components contained in $s$ fibres of $|2F|$ is at most $8+s$.
\end{proposition}
\begin{proof}
We pass to the Jacobian fibration of $|2F|$, which is an elliptic fibration $p:T\rightarrow \BP^1$ with a section. According to \cite[Corollary~4.3.18, Theorem~4.3.20]{EnriquesI}, $T$ is an elliptic surface of the same Picard rank and the types of singular fibres of $p$ coincide with those of $|2F|$. Therefore it suffices to prove the statement for the elliptic surface $p:T\rightarrow \BP^1$, and this follows from the Shioda--Tate formula \cite[Corollary~6.13]{SchuettShioda}.
\end{proof}

\begin{definition}
An \emph{(isotropic) $n$-sequence} is a list of $n$ isotropic (i.e.\ of square $0$) divisors
$$
\big(F_1,F_1 + R_{1,1},\hdots,F_1 + \sum_{j=1}^{r_1} R_{1,j},F_2,\ldots,F_2 + \sum_{j = 1}^{r_2}R_{2,j},\hdots, F_c, \ldots, F_{c} + \sum_{j=1}^{r_c} R_{c,j}\big),
$$
where the \(F_i\) are half-fibres on \(S\) and the \(R_{i,j}\) are \((-2)\)-curves satisfying the following conditions:
\begin{enumerate}
    \item \(F_i.F_j = 1 - \delta_{ij}\).
    \item \(R_{i,j}.R_{i,j+1} = 1\).
    \item \(R_{i,j}.R_{k,l} = 0\) unless \((k,l) = (i,j)\) or \((k,l)=(i,j\pm 1)\).
    \item \(F_i.R_{i,1} = 1\) and \(F_i.R_{k,l} = 0\) if \((k,l) \neq (i,1)\).
\end{enumerate}
We refer to the chain of $(-2)$-curves $R_{i,1},\ldots,R_{i,r_i}$ as the \emph{tail} of $F_i$.
\end{definition}

Notice that by definition the curves $R_{i,1},\ldots,R_{i,r_i}$ are vertical components of reducible fibres of the elliptic pencils $|2F_j|$ for $j\ne i$.

If $(E_1,\ldots,E_n)$ is an $n$-sequence, then $E_i.E_j=1-\delta_{ij}$ for any $1\le i,j \le n$.
The half-fibres appearing in an $n$-sequence satisfy several properties:

\begin{proposition}[{\cite[Proposition~2.8]{MMV}}] \label{prop:half_fibres}
Let $F_1,F_2$ be half-fibres with $F_1.F_2=1$.
\begin{enumerate}
    \item $F_1$ and $F_2$ share no irreducible component.
    \item There exists a simple fibre $G_2\in |2F_2|$ containing all components of $F_1$ except one.
\end{enumerate}
\end{proposition}

We conclude the section with some crucial results concerning $n$-sequences, that exhibit the rich structure which is present on Enriques surfaces. In the following theorems concluding the section, we assume that $S$ is an Enriques surface over the algebraically closed field $\ov k$ (of characteristic different from $2$).

\begin{theorem}[{\cite{Cossec}}] \label{thm:extendability}
 Every $n$-sequence $(E_1,\ldots,E_n)$ with $n<9$ can be extended to a $10$-sequence. In other words there exist isotropic divisors $E_{n+1},\ldots,E_{10}$ such that $(E_1,\ldots,E_{10})$ is a $10$-sequence up to reordering.
\end{theorem}

\begin{proof}
Pick ten vectors $e_1,\ldots,e_{10}\in \Num(S)$ such that $e_i^2=0$ for any $i=1,\ldots,10$ and $e_i.e_j=1$ for any $i\ne j$ (e.g. follow \cite[§1.1]{Cossec}). By \cite[Lemma~1.6.1]{Cossec} there exists an isometry $ \psi$ of $\Num(S)$ such that $\psi(e_i)=E_i$ for $i=1,\ldots,n$, so we can define $E_j':=\psi(e_j)$ for $j>n$ and obtain a $10$-uple $(E_1,\ldots,E_n,E_{n+1}',\ldots,E_{10}')$ such that $E_i^2=E_i'^2=0$ for any $i=1,\ldots,10$ and $E_i.E_j=E_i.E_j'=E_i'.E_j'=1$ for any $i\ne j$.
Now we apply \cite[Theorem~3.3]{Cossec} and we obtain the desired $10$-sequence $(E_1,\ldots,E_{10})$. Notice that by the uniqueness of the divisors $E_1,\ldots,E_{10}$ in \cite[Theorem~3.3]{Cossec}, it follows that this process does not affect the $n$ divisors $E_1,\ldots,E_n$.
\end{proof}

\begin{theorem}[{\cite[Theorem~3.5]{Cossec}, \cite[Theorem~1.3, Corollary~1.6]{MMV0}}]\label{thm:three}
 Let $S$ be an Enriques surface and $F_1$ be a half-fibre on $S$. There exist half-fibres $F_2$ and $F_3$ such that $(F_1,F_2,F_3)$ is a $3$-sequence.
\end{theorem}

\begin{theorem}[{\cite[Corollary~1.4]{MMV}}] \label{thm:four}
 Every Enriques surface admits four half-fibres $F_1,\ldots,F_4$ forming a $4$-sequence.
\end{theorem}

\section{Overlattices of root lattices}\label{sec:moduli}

We devote this section to an in-depth study of the overlattices of root lattices. In order to prove the Hilbert property for many K3 surfaces, we need to show that the complement of a union $Z$ of $(-2)$-curves is simply connected. From \cite[Theorem~4.3]{ShimadaZhang} by Shimada and Zhang, it is sufficient to study the primitiveness of the embedding of the root lattice $R$ spanned by the components of $Z$ into the Picard lattice of the K3 surface. The goal of this section is to show that, if this embedding is \emph{not} primitive, then $R$ satisfies several restrictive properties.
We will keep the notations and conventions about ADE lattices introduced in Section \ref{sec:lattices}.

The following result already appeared in \cite[Corollary 1.2]{Schuett}; we give a simpler proof using the rigidity of $(-2)$-curves on K3 surfaces. Recall that, if $L\hookrightarrow M$ is an embedding of lattices, the saturation $L_{sat}$ of $L$ in $M$ is the smallest primitive sublattice of $M$ containing the image of $L$.

\begin{proposition} \label{prop:primitive}
Let $X$ be a K3 surface, and consider a root lattice $R$ spanned by some $(-2)$-curves on $X$. The roots of the saturation $R_{sat}$ in $\Pic(X)$ belong to $R$. In particular, if the embedding $R\hookrightarrow \Pic(X)$ is not primitive, then any intermediate overlattice $R\subsetneq R'\subseteq R_{sat}$ is not a root lattice.
\end{proposition}
\begin{proof}
Let $D$ be a vector of norm $-2$ in the saturation $R_{sat}$ of $R$.
If $R$ is spanned by $(-2)$-curves $C_1,\ldots,C_r$, we can write $D=\sum_{i}{\alpha_i C_i}$ for $\alpha_i \in \QQ$. By Riemann--Roch either $D$ or $-D$ is effective, and without loss of generality we can assume that $D$ is effective. Hence at least one of the $\alpha_i$ is positive. Moreover we can write $D=C_1'+\ldots+C_s'$ for some (not necessarily distinct) irreducible curves $C_1',\ldots,C_s'$ on $X$. After separating the $C_i$ with positive and negative coefficients $\alpha_i$, we get an equality
\begin{equation} \label{eq:1}
\sum_j{C_j'}+\sum_{\alpha_i < 0}{(-\alpha_i)C_i}=\sum_{\alpha_i \ge 0}{\alpha_i C_i}
\end{equation}
inside $\Pic(X)$. Up to multiplying both sides by a positive integer, we can assume that the $\alpha_i$ are integers. Since the curves $C_i$ belong to the negative definite lattice $R$, the linear system associated to the right hand side contains only one effective divisor, which is itself. Indeed for any decomposition $\sum_{\alpha_i \ge 0}{\alpha_i C_i}=M+F$ with $M$ mobile and effective and $F$ fixed and effective, we have $M\in R$ (since clearly the fixed locus $F$ belongs to $R$) and therefore $M^2 < 0$ unless $M=0$.

Since the left hand side is effective as well, we deduce that equality (\ref{eq:1}) is an equality of divisors. This implies that the curves $C_1',\ldots,C_s'$ coincide with some of the $C_i$, hence they are contained in $R$, and so $D\in R$.
\end{proof}

\begin{proposition} \label{prop:non_root_overlattice}
Let $R$ be a root lattice and let $S$ be the isotropic subgroup of $A_R$ corresponding to the overlattice $R'$ of $R$. The overlattice $R'$ is a root lattice if and only there is a set $\{v_1,\ldots,v_n\}$ of generators of $S$ such that each $v_i$ has a preimage $w_i\in R^\vee$ of norm $-2$ under the projection map $R^\vee\to A_R$.
\end{proposition}
\begin{proof}
By definition $R'$ is obtained from $R$ by adjoining preimages $w_i \in R^\vee$ of the $v_i\in A_R$. If the $w_i$ have norm $-2$, then $R'$ admits a set of generators given by vectors of norm $-2$, and hence $R'$ is a root lattice.

Assume conversely that $R'$ is a root lattice, and let $b_1,\ldots,b_m$ be a set of roots generating $R'$. Up to permuting the indices we can assume that the projections $v_1=\overline b_1,\ldots, v_n=\overline b_n\in A_R$ of $b_1,\ldots,b_n$ under the projection map $R^\vee\to A_R$ generate the subgroup $S$. By construction $\{v_1,\ldots,v_n\}$ is the desired set of generators for $S$.
\end{proof}

Let $R$ be a root lattice, spanned by vectors $e_i$, and $R'$ an overlattice of $R$ of index $p$ prime. We say that $R'$ is \emph{given} by $v\in R^\vee$ if $R'=R[v]$, that is, if $R'$ is obtained from $R$ by adjoining $v$. Moreover, we consider $v$ to be \emph{minimal}, i.e.\ $v=\frac{1}{p}(\alpha_1e_1+\ldots+\alpha_re_r)$ for some $0\le \alpha_i\le p-1$. We say that the overlattice $R'$ is \emph{concentrated} at a sublattice $\widetilde{R}\subseteq R$ if $\widetilde{R}$ is spanned by those $e_i$ with a non-zero coefficient $\alpha_i$ in $v$. Notice that this definition depends on the choice of the generators of $R$. The following proposition shows that overlattices of root lattices are concentrated at very special sublattices.

\begin{proposition} \label{prop:conc}
Let $R$ be a root lattice, $R'$ an overlattice of $R$ of index $p$ prime. Then $R'$ is concentrated at a sublattice of $R$ isometric to $\mathbf{A}_{p-1}^{\oplus r}$, for some $r\ge 1$. Therefore, if $R'$ is not a root lattice, then $\mathbf{A}_{p-1}^{\oplus r}$ admits an overlattice that is not a root lattice as well.
\end{proposition}
\begin{proof}
Assume that $R=R_1\oplus \ldots\oplus R_N$ splits as the direct sum of certain ADE lattices. The overlattice $R'$ of $R$ is given by a minimal vector $v\in R^\vee$, and its image $\overline{v}\in A_R =A_{R_1}\times \ldots\times A_{R_N}$ has order $p$, so $\overline{v}=(\overline{v_1},\ldots,\overline{v_N})$ for certain $\overline{v_i}\in A_{R_i}$ of order $p$ (or $\overline{v_i}=0$). Therefore we only need to classify the vectors of order $p$ in $A_R$, where $R$ is an ADE lattice of discriminant multiple of $p$.

Assume $R=\mathbf{A}_{pr-1}$, spanned by $e_1,\ldots,e_{pr-1}$ such that $e_i.e_{i+1}=1$ for each $i$. The discriminant group $A_R$ is cyclic of order $pr$ by Table \ref{ADE}, and an easy verification shows that its generator is the image of $e=\frac{1}{pr}(\sum_{j=1}^{pr-1}{je_j})\in R^\vee$. All vectors of order $p$ in $A_R$ are the images of $mr\cdot e\in R^\vee$, for $m\in \{1,\ldots,p-1\}$. It is easy to notice that, after minimalising $mr\cdot e$, its coefficients are zero precisely in the positions $p,2p,\ldots,(r-1)p$. Thus $mr\cdot e\in R'$ is concentrated at the curves $\{e_i : p\nmid i\}$, that span a sublattice of $R$ isometric to $\mathbf{A}_{p-1}^{\oplus r}$.

The cases $R=\mathbf{D}_n$ and $R=\mathbf{E}_6, \ \mathbf{E}_7, \ \mathbf{E}_8$ are treated analogously.

For the second part of the statement, assume that $R'$ is not a root lattice, and that it is given by a minimal vector $v=\frac{1}{p}(\alpha_1e_1+\ldots+\alpha_re_r)\in R^\vee$. $R'$ is concentrated at a sublattice $\widetilde{R}\subseteq R$ isometric to $\mathbf{A}_{p-1}^{\oplus r}$, hence $v\in \widetilde{R}^\vee$ induces an overlattice $\widetilde{R}'$ of $\widetilde{R}$ of index $p$. If by contradiction $\widetilde{R}'$ is a root lattice, then there exists a root $w$ in $\widetilde{R}'$ that is not in $\widetilde{R}$. The root $w$ belongs to $R'$ as well, since it is an integral linear combination of some of the $C_i$ and $v$. Moreover $w$ does not belong to $R$, since otherwise it would be in $\widetilde{R}=\widetilde{R}' \cap R$. We obtain that $w$ is a root in $R'\setminus R$, and hence $R[w]$ is a root lattice such that $R\subsetneq R[w] \subseteq R'$. As the index of $R\subseteq R'$ is prime, we conclude that $R'=R[w]$ is a root lattice, contradicting the assumption.
\end{proof}

\begin{remark} \label{rk:geometric}
We can interpret Proposition \ref{prop:conc} in a more geometric way, at least when $R$ is a root lattice generated by $(-2)$-curves on a K3 surface $X$. If the embedding $R\hookrightarrow \Pic(X)$ is not primitive, then the saturation $R_{sat}$ contains an overlattice $R'\subseteq\Pic(X)$ of $R$ of index $p$ prime. We denote by $X_0$ the singular K3 surface obtained from $X$ by contracting all the $(-2)$-curves in $R$ where the overlattice $R'$ is concentrated, and by $t_1,\ldots,t_r$ the resulting singular points of $X_0$ (notice that we are able to make the contraction since $R$ is negative definite \cite[Theorem~2.3]{artin.contractability}). The fact that $R'\subseteq \Pic(X)$ implies by \cite[\S1.1]{KeumZhang} that there exists a cyclic $(p:1)$-cover $X'\to X_0$ branched precisely over the singular points $t_1,\ldots,t_r\in X_0$, where $X'$ is smooth. Since the image of a smooth ramification point of a cyclic $(p:1)$-cover is a singular point of type $A_{p-1}$ (see \cite[Proposition III.5.3]{BarthHulek}), we conclude that $R'$ is concentrated at a root lattice isometric to $\mathbf{A}_{p-1}^{\oplus r}$.
\end{remark}

\begin{lemma}\label{lem:overnonroot}
Let $R=\mathbf{A}_{p-1}^{\oplus r}$ for some prime $p$ and $r\ge 1$, and assume that $R$ admits an overlattice (resp. an overlattice that is not a root lattice).
We have:
\begin{itemize}
    \item $\rk R\ge 4$ (resp.\ $\rk R\ge 8$) when $p=2$;
    \item $\rk R\ge 6$ (resp.\ $\rk R\ge 12$) when $p=3$;
    \item $\rk R\ge 8$ (resp.\ $\rk R\ge 16$) when $p\ge 5$.
\end{itemize}
\end{lemma}
\begin{proof}
Let $v\in A_R$ be an isotropic vector, giving the overlattice $R'$ of $R$. By \Cref{prop:non_root_overlattice}, $R'$ is a root lattice if and only if there is a preimage of $v$ in $R^\vee$ of norm $-2$.

Consider first $R=\mathbf{A}_1^{\oplus r}$. The discriminant group $A_R$ is isomorphic to $(\ZZ/2)^{\oplus r}$, and a set of generators for $A_R$ is given by orthogonal vectors $v_1,\ldots,v_r$ of norm $-\frac{1}{2} \pmod{2}$. Thus the isotropic vectors in $A_R$ are the sum of $4m$ vectors among the $v_i$, for some $m\ge 1$. 
We deduce that $R$ has an overlattice if and only if $r\ge 4$.
Moreover, since the norm of any preimage of the $v_i$ in $R^\vee$ is $\le - \frac{1}{2}$, any preimage of the vector $v=\sum_{i=1}^{4m}{v_i}$ has norm $\le -2m$. 
In particular $R$ admits an overlattice that is not a root lattice if and only if $r\ge 8$.

Similarly one can check that $\mathbf{A}_2^{\oplus r}$ admits an overlattice (resp. an overlattice that is not a root lattice) if and only if $r\ge 3$ (resp. $r\ge 6$), and $\mathbf{A}_4^{\oplus r}$ admits an overlattice (resp. an overlattice that is not a root lattice) if and only if $r\ge 2$ (resp. $r\ge 4$).
We conclude by noticing that the root lattices $\mathbf{A}_6$, $\mathbf{A}_6^{\oplus 2}$, $\mathbf{A}_{10}$ and $\mathbf{A}_{12}$ have no overlattice.

\end{proof}

We conclude the section with an easy application of \Cref{prop:conc} and \Cref{lem:overnonroot}.

\begin{proposition} \label{prop:overlattice_ge8}
If $R$ is a root lattice admitting an overlattice that is not a root lattice, then $\rk R\ge 8$.
\end{proposition}
\begin{proof}
Let $R$ be a root lattice and $R'$ an overlattice of $R$ that is not a root lattice.
Consider the lattice $R'_{root}$ spanned by the roots of $R'$. We have the inclusions $R\subseteq R'_{root}\subsetneq R'$. Choose any lattice $R''$ such that $R'_{root}\subsetneq R''\subseteq R'$ and the index $[R'':R'_{root}]$ is prime. By construction $R''$ is not a root lattice, hence up to exchanging $R$ with $R_{root}'$ and $R'$ with $R''$, we may assume in the statement that the index $[R':R]$ is a prime $p$.
By \Cref{prop:conc} the overlattice $R'$ is concentrated at a sublattice $\widetilde R \subseteq R$ isometric to $\mathbf{A}_{p-1}^{\oplus r}$, and $\widetilde R$ admits an overlattice that is not a root lattice as well. It then suffices to prove the statement for the root lattices of the form $\mathbf{A}_{p-1}^{\oplus r}$ with $p$ prime, and this follows from \Cref{lem:overnonroot}.
\end{proof}

\section{Exceptional and over-exceptional lattices}\label{sec:exceptional}
This section introduces and proves properties of the central object of this article, the \emph{over-exceptional} lattice.

\begin{definition}\label{def:exceptional}
 Let $X$ be a minimal algebraic surface of Kodaira dimension $0$ over an algebraically closed field. The \emph{over-exceptional lattice} $E'(X)$ (resp. \emph{exceptional lattice } $E(X)$) of $X$ is the sublattice of $\Num(X)$ spanned by the classes of the irreducible curves on $X$ which are orthogonal to the fibres of every elliptic fibration (resp. every elliptic fibration with infinite stabilizer in $\Aut(X)$) on $X$.
\end{definition}

The irreducible curves contained in fibres of a relatively minimal elliptic fibration are either curves of arithmetic genus $1$, if the fibre is irreducible, or smooth rational curves, if the fibre is reducible. By the genus formula they have square $0$ and $-2$ respectively. Since an irreducible curve of arithmetic genus $1$ is only orthogonal to the elliptic fibration induced by itself, the over-exceptional (resp. exceptional) lattice is generated by classes of $(-2)$-curves as soon as $X$ contains two distinct elliptic fibrations (resp. two distinct elliptic fibrations with infinite stabilizer in $\Aut(X)$).

\begin{remark}\label{rem:ENeqE}
Obviously $E'(X)\subseteq E(X)$. The exceptional lattice $E(X)$ of a K3 surface $X$ over an algebraically closed field of characteristic $0$ has been studied previously by Nikulin in \cite{Nikulin}. His definition \cite[\S4, p.\ 258]{Nikulin} is slightly different from \Cref{def:exceptional} but agrees when $\rho(X)\geq 6$. Indeed, when $X$ has an elliptic fibration with infinite stabiliser (or equivalently when $X$ has infinite automorphism group, see \cite[Theorem~3.1]{Nikulin}), this follows from \cite[Theorem~4.1]{Nikulin}; otherwise, Nikulin's exceptional lattice and $E(X)$ are equal to $\Pic(X)$ by \cite[Definition~4.5]{Nikulin}.
\end{remark}

We will be mainly interested in studying the over-exceptional lattices of Enriques and K3 surfaces, in view of applications towards the Hilbert Property. 

\begin{example} \label{ex:kummer}
 \begin{enumerate}
  \item Let $S$ be an Enriques surface with finite automorphism group over an algebraically closed field of characteristic different from $2$. These have been classified into seven types by Kond\=o and Martin in \cite{Kondo, Martin}, together with their $(-2)$-curves and elliptic fibrations. Going through the list, one checks that $E'(S)=0$ while $E(S)=\Num(S)$.
  \item Let $X$ be the Kummer surface associated to the product of two non-isogenous elliptic curves over an algebraically closed field of characteristic $0$. Then it is shown in \cite[Remark~6.3(iii)]{Yu} that $E(X)\cong\BZ^8$. However, using an extra elliptic fibration from \cite[Lemma~1.1]{KuwataWang} with finite stabiliser, it follows that $\rk E'(X)\leq 6$.
 \end{enumerate}
\end{example}

\begin{lemma} \label{lemma:overexceptionalE6}
Let $S$ be an Enriques surface over an algebraically closed field of characteristic different from $2$. The over-exceptional lattice $E'(S)$ of $S$ is a root lattice and it embeds into $\mathbf{E}_6$.
\end{lemma}
\begin{proof}
Let $F_1,F_2,F_3,F_4$ be half-fibres on $S$ forming a $4$-sequence as in \Cref{thm:four}. The lattice $\langle F_1,F_2,F_3,F_4\rangle$ spanned by the four half-fibres is isometric to $\mathbf{U}\oplus \mathbf{A}_2$ (if $\mathbf{U}$ is generated by $F_1$ and $F_2$, its orthogonal complement, spanned by $F_1+F_2-F_3$ and $F_3-F_4$, is isometric to $\mathbf{A}_2$). 
We claim that the orthogonal complement in $\Num(S)$ of the lattice $\langle F_1,F_2,F_3,F_4\rangle$ is isometric to $\mathbf{E}_6$. Indeed the isometry class of the orthogonal complement does not depend on the embedding $\mathbf{U}\oplus \mathbf{A}_2 \hookrightarrow \Num(S)=\mathbf{U}\oplus \mathbf{E}_8$ by \cite[Theorem 1.14.2]{Nikulin_3}, hence we only need to exhibit an embedding $\mathbf{A}_2\hookrightarrow \mathbf{E}_8$, whose orthogonal complement is isometric to $\mathbf{E}_6$. In order to achieve this, notice that there is an isotropic vector of order $3$ in the discriminant group of $\mathbf{A}_2\oplus \mathbf{E}_6$, so there is an overlattice $L$ of $\mathbf{A}_2\oplus \mathbf{E}_6$ of discriminant $1$. Since $\mathbf{E}_8$ is the unique even, negative definite, unimodular lattice of rank $8$, we have that $L$ is isometric to $\mathbf{E}_8$, and this gives the desired primitive embedding $\mathbf{A}_2\hookrightarrow \mathbf{E}_8$.

Now by definition the curves in the over-exceptional lattice are orthogonal to the $F_i$, so we obtain an embedding $E'(S)\hookrightarrow \langle F_1,F_2,F_3,F_4\rangle^\perp =\mathbf{E}_6$. In particular $E'(S)$ is negative definite, and since it is spanned by $(-2)$-curves, we deduce that $E'(S)$ is a root lattice.
\end{proof}

In the following let $S$ be an Enriques surface over an algebraically closed field of characteristic different from $2$. We choose by \Cref{thm:four} four half-fibres forming a $4$-sequence $(F_1,F_2,F_3,F_4)$ on $S$, and by \Cref{thm:extendability} we extend it to a $10$-sequence 
\[(F_1,F_2,F_3,F_4,E_1,\ldots,E_6).\]
The $E_i$ are either half-fibres meeting $F_1,\ldots,F_4$ with multiplicity $1$, or divisors of the form $F+D$, with $F$ a half-fibre in the $10$-sequence and $D$ a chain of $(-2)$-curves, contained in the tail of $F$.

In particular the $10$-sequence contains at most six $(-2)$-curves. Some of them meet a half-fibre in the $10$-sequence with multiplicity $1$, while the other ones, of the form $E_{j+1}-E_j$ for some $1\le j\le 5$, are orthogonal to all the half-fibres in the $10$-sequence.

We denote $H=\frac{1}{3}(F_1+\ldots+F_4+E_1+\ldots+E_6)\in \Num(S)$. The divisor $H$ is integral by \cite[Lemma~1.6.2]{Cossec}.

\begin{lemma} \label{lemma:possible_curves}
Any curve $C$ in the over-exceptional lattice $E'(S)$ of $S$ coincides with one of the following divisors:
\begin{enumerate}
    \item $R_j := E_{j+1}-E_j$;
    \item $R_{j_1,j_2,j_3} := H-E_{j_1}-E_{j_2}-E_{j_3}$, and $R_{j_1,j_2,j_3}.E_j=1$ if $j\in \{j_1,j_2,j_3\}$, while $R_{j_1,j_2,j_3}.E_j=0$ otherwise;
    \item $R := 2H - \sum_{j=1}^6{E_i}$, and $R.E_j = 1$ for each $1\le j\le 6$.
\end{enumerate}
Moreover
$$R.R_j=0, \quad R.R_{j_1,j_2,j_3}=-1, \quad R_j.R_{j'} = \begin{cases}
0 &\text{if } |j-j'|\ge 2\\
1 &\text{if } |j-j'|=1
\end{cases}$$
$$R_{j_1,j_2,j_3}.R_j = \#\{j+1,j_1,j_2,j_3\}-\#\{j,j_1,j_2,j_3\}$$
$$R_{j_1,j_2,j_3}.R_{j_4,j_5,j_6}=\#\{j_i\mid i=1,\dots,6\}-5$$
\end{lemma}
\begin{proof}
If $C$ is a $(-2)$-curve in the $10$-sequence, necessarily it coincides with one of the $R_j$, since it is orthogonal to all the half-fibres in the $10$-sequence. Hence we may assume that $C$ does not belong to the $10$-sequence, and in particular that $C.E_j\ge 0$ for each $1\le j\le 6$.

We claim that $0\le C.E_j \le 1$ for each $1\le j\le 6$. If $E_j$ is a half-fibre, then $C.E_j=0$ by definition. Otherwise we can write $E_j = F + D$, with $F$ a half-fibre in the $10$-sequence and $D$ a chain of $(-2)$-curves orthogonal to at least three of the $F_i$, say $F_1$, $F_2$ and $F_3$. In particular there exists a fibre $G_1\in |2F_1|$ containing the chain $D$. Since by assumption $C$ is contained in a fibre of $|2F_1|$, necessarily $C.E_j=C.D\le 2$. Moreover, if $C.D=2$, then $C$ forms an elliptic fibre of type $I_m$ together with the curves in $D$, and this fibre must coincide with $G_1$. But then $G_1.F_2 = (C+D).F_2 = 0$, a contradiction.

Therefore $C.E_j\in \{0,1\}$ for each $1\le j\le 6$. Since $C.H$ is an integer, $C$ intersects either $3$ or $6$ of the $E_i$. Since the divisors in the $10$-sequence span $\Num(S)$ over $\BQ$, there is at most one divisor in $\Num(S)$ with prescribed intersections with the $F_i$ and $E_j$. It is then easy to check that the vectors $R_{j_1,j_2,j_3}$ and $R$ have self-intersection $-2$ and the desired intersection numbers with the $F_i$ and $E_j$.
\end{proof}

\begin{proposition} \label{prop:small_overexceptional}
Let $S$ be an Enriques surface over an algebraically closed field of characteristic different from $2$. The over-exceptional lattice $E'(S)$ of $S$ does not contain the following configurations of $(-2)$-curves:
\begin{itemize}
    \item four disjoint $(-2)$-curves;
    \item six $(-2)$-curves spanning a lattice isometric to $\mathbf{A}_2^{\oplus 3}$;
    \item six $(-2)$-curves spanning a lattice isometric to $\mathbf{E}_6$.
\end{itemize}
\end{proposition}

We have the following first reduction step.

\begin{lemma}\label{lem:simple_config}
Assume that the over-exceptional lattice $E'(S)$ contains one of the three configurations of $(-2)$-curves in \Cref{prop:small_overexceptional}. Then there exists an elliptic pencil $|2F|$ on $S$ and a simple fibre $G\in |2F|$ containing the same configuration of $(-2)$-curves.
\end{lemma}
\begin{proof}
A configuration of $(-2)$-curves of type $\mathbf{E}_6$ is connected, hence if such a configuration belongs to the over-exceptional lattice $E'(S)$, by definition it is contained in a fibre of every elliptic fibration on $S$. The fibre must be simple, since any fibre containing a vertex of valency $3$ is additive. Hence we only have to consider the other two cases, which will be dealt with separately.
We will show first that there is an elliptic fibre $G$ containing the configuration of $(-2)$-curves, and at the end that we can take $G$ to be a simple fibre.
\begin{enumerate}
    \item Assume that $E'(S)$ contains four disjoint $(-2)$-curves $C_1,C_2,C_3,C_4$ and consider the $10$-sequence $(F_1,F_2,F_3,F_4,E_1,\ldots,E_6)$. Since the $C_i$ are orthogonal to the $F_i$, they cannot be the first curve of any tail in the $10$-sequence. In particular the $C_i$ cannot all belong to the $10$-sequence, so at least one of them, say $C_1$, does not belong to the $10$-sequence.
    
    Assume first that $C_1=R$. Since $R.R_{j_1,j_2,j_3}=-1<0$ for any $j_1,j_2,j_3$, the divisors $R_{j_1,j_2,j_3}$ are not $(-2)$-curves. In particular $C_2,C_3,C_4$ must belong to the $10$-sequence, and they are contained in a maximum of $3$ tails. The $(-2)$-curve $R$ intersects the first curve of each tail of the $10$-sequence, so the tails containing $C_2,C_3,C_4$ together with $C_1=R$ form a connected configuration of $(-2)$-curves orthogonal to one of the $F_i$, so there is a fibre $G_i \in |2F_i|$ containing it.
    
    Assume instead that $C_1=R_{j_1,j_2,j_3}$. By \Cref{lemma:possible_curves} there are at most two of the $R_j$ that are disjoint and orthogonal to $C_1$, so we may assume again by \Cref{lemma:possible_curves} that $C_2=R_{j_3,j_4,j_5}$.
    
    We claim that $C_3$ and $C_4$ belong to the $10$-sequence. If by contradiction $C_3=R_{j_1,j_4,j_6}$, then by \Cref{lemma:possible_curves} we would necessarily have $C_4=R_{j_2,j_5,j_6}$. Each $E_j$ intersects at least one of the $C_i$, so the $10$-sequence contains precisely $4$ half-fibres. However each $R_j$ intersects negatively at least one of the $C_i$, so the divisors $R_j$ for $1\le j\le 5$ are not effective.
    In particular all the tails in the $10$-sequence have length $\le 1$, contradicting the fact that there are $10$ divisors in the $10$-sequence.
    
    Therefore $C_3$ and $C_4$ belong to the $10$-sequence, and since the $C_i$ are all disjoint, by \Cref{lemma:possible_curves} we must have $C_1=R_{j_1,j_1+1,j_3}$, $C_2=R_{j_3,j_4,j_4+1}$, $C_3=R_{j_1}$ and $C_4=R_{j_4}$. The curve $C_1$ intersects $E_{j_1}$ and $E_{j_3}$, so it intersects a curve in the tail containing $C_3=E_{j_1+1}-E_{j_1}$ and the tail containing $E_{j_3}$, while $C_2$ intersects the tail containing $C_4$ and the tail containing $E_{j_3}$. In particular the tails containing $C_3$, $C_4$ and $E_{j_3}$, together with $C_1$ and $C_2$, form a connected configuration of $(-2)$-curves orthogonal to one of the $F_i$, so there is a fibre $G_i \in |2F_i|$ containing it.
    
    \item Assume that $E'(S)$ contains six $(-2)$-curves $C_1,\ldots,C_6$ spanning a lattice isometric to $\mathbf{A}_2^{\oplus 3}$.
    Up to renumbering we can assume that $C_1.C_2=C_3.C_4=C_5.C_6=1$, and that the other intersection numbers are $0$.
    We consider the $10$-sequence $(F_1,F_2,F_3,F_4,E_1,\ldots,E_6)$. The six curves $C_i$ cannot all be contained in the $10$-sequence, so we may assume that $C_1$ does not belong to the $10$-sequence.
    The curve $C_1$ cannot coincide with the divisor $R$ in \Cref{lemma:possible_curves}, since no curve in \Cref{lemma:possible_curves} intersects $R$ with multiplicity $1$. Hence we may assume $C_1=R_{j_1,j_2,j_3}$.
    
    If $C_2$ does not belong to the $10$-sequence, the equation $C_1.C_2=1$ implies by \Cref{lemma:possible_curves} that $C_2=R_{j_4,j_5,j_6}$ for a certain permutation $(j_1,\ldots,j_6)$ of $(1,\ldots,6)$. The other $4$ curves $C_3,\ldots,C_6$, being orthogonal to $C_1$ and $C_2$, must belong to the $10$-sequence, and more precisely $C_3=R_1$, $C_4=R_2$, $C_5=R_4$, $C_6=R_5$. Up to switching $C_1$ and $C_2$, this forces $C_1=R_{1,2,3}$ and $C_2=R_{4,5,6}$. Since $C_1$ intersects the tail containing $C_3$ and $C_4$, and $C_2$ intersects the tail containing $C_5$ and $C_6$, the two tails together with $C_1$ and $C_2$ form a connected configuration of $(-2)$-curves orthogonal to one of the $F_i$.
    
    Hence we may assume that in each pair $(C_1,C_2)$, $(C_3,C_4)$, $(C_5,C_6)$, at least one of the two curves belongs to the $10$-sequence, say $C_2,C_4$ and $C_6$. The curve $C_3$ does not belong to the $10$-sequence, since otherwise it would be impossible to fit the curves $C_2,C_3,C_4,C_6$ (spanning a lattice isometric to $\mathbf{A}_1\oplus \mathbf{A}_2\oplus \mathbf{A}_1$) among the (at most) six $(-2)$-curves in the $10$-sequence. 
    The same argument applies to $C_5$, hence $C_1$, $C_3$ and $C_5$ do not belong to the $10$-sequence. This forces the three disjoint curves $C_2,C_4,C_6$ to coincide with the following divisors: $C_2=R_1$, $C_4=R_3$, $C_6=R_5$. Consequently $C_1=R_{2,3,4}$, $C_3=R_{4,5,6}$, $C_5=R_{1,2,6}$. Since $C_1$ (resp. $C_3$, $C_5$) intersects the tails containing $C_2$ and $C_4$ (resp. $C_4$ and $C_6$, $C_2$ and $C_6$), we have that the tails containing $C_2$, $C_4$, $C_6$, together with $C_1$, $C_3$ and $C_5$, form a connected configuration of $(-2)$-curves orthogonal to one of the $F_i$.
\end{enumerate}

It only remains to show that the fibre $G\in |2F|$ containing the $C_i$ can be assumed to be simple. If $G=2F$ is a multiple fibre, choose by \Cref{thm:extendability} a half-fibre $F'$ with $F.F'=1$. By \Cref{prop:half_fibres} there exists a simple fibre $G'\in |2F'|$ containing all components of $G$ except one, say $C$. This component satisfies $C.G'>0$, so it is not one of the $C_i$. It follows that $G'$ is a simple fibre containing all the $C_i$, as claimed.

\end{proof}

\begin{proof}[Proof of \Cref{prop:small_overexceptional}]
We argue by contradiction, so assume that the over-exceptional lattice of $S$ contains a configuration of $(-2)$-curves as in \Cref{prop:small_overexceptional}.
By \Cref{lem:simple_config}, let $G\in |2F|$ be the simple fibre containing the configuration of curves, which we denote by $C_i$. By \Cref{thm:extendability}, let $F'$ and $F''$ be half-fibres such that $(F,F',F'')$ is a $3$-sequence. The curve $C_1$ is orthogonal to both $F'$ and $F''$ by assumption, and since half-fibres of $F'$ and $F''$ do not share components by \Cref{prop:half_fibres}, without loss of generality we can assume that $C_1$ belongs to a simple fibre $G'\in |2F'|$. Therefore by \cite[Remark~3.10]{MMV} we can extend $(F,F')$ to a \emph{special} $3$-sequence $(F,F',F'')$ (cf. \cite[Definition~3.1]{MMV}).
More precisely, the simple fibres $G$ and $G'$ (and hence the curves $C_i$) belong to the \emph{triangle graph} (cf. \cite[Definition~3.5]{MMV}) of the $3$-sequence, and by \cite[Proposition~3.7]{MMV} there are only finitely many possible structures for the triangle graph, listed in \cite[Table~1]{MMV}.

In particular the curves $C_i$ belong to an elliptic fibre in one of the dual graphs in \cite[Table~1]{MMV}, and they are orthogonal to all the elliptic fibres supported on the graph. 

For most of the graphs, it is straightforward to check that this is contradictory: we simply list all the elliptic fibres on the graph, and see that there are no configurations of $(-2)$-curves as in the statement that are orthogonal to all the elliptic fibres.

Let us discuss the non obvious cases here below. Notice that the three elliptic fibres obtained as $S_i+S_j$ in \cite[Table~1]{MMV} are simple by \cite[Proposition~3.3.(3)]{MMV}.

In the case $(\mathbf{E}_8,\mathbf{A}_1,\mathbf{A}_1)$, there is a configuration of type $\mathbf{E}_7$ orthogonal to the simple fibre $G_1$ with $2$ components. This configuration is contained in a reducible fibre $G_1'\in|G_1|$, which by \Cref{prop:embedE8} must be of type $III^*$. Since the bisection of $G_1$ is also a bisection of $G_1'$, we obtain the following graph:

$$
    
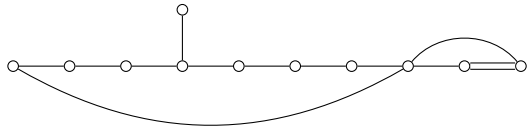
\begin{tikzpicture}[scale=0.75]
    \node (R0) at (-1,0) [nodal,fill=white] {};
\node (R1) at (0,0) [nodal,fill=white] {};
\node (R2) at (1,0) [nodal,fill=white] {};
\node (R3) at (2,0) [nodal,fill=white] {};
\node (R4) at (3,0) [nodal,fill=white] {};
\node (R5) at (4,0) [nodal,fill=white] {};
\node (R6) at (5,0) [nodal,fill=white] {};
\node (R7) at (6,0) [nodal,fill=white] {};
\node (R8) at (7,0) [nodal,fill=white] {};
\node (R9) at (8,0) [nodal,fill=white] {};
\node (R10) at (2,1) [nodal,fill=white] {};

\draw (R0)--(R1)--(R2)--(R3)--(R4)--(R5)--(R6)--(R7)--(R8) (R3)--(R10) (R7) to[bend left=50] (R9) (R0) to[bend right] (R7);
\draw[double] (R8)--(R9);
    \end{tikzpicture}
$$

It is straightforward to check that this graph has no vertex orthogonal to all elliptic configurations.

In the case $(\mathbf{E}_7,\mathbf{A}_1,\mathbf{A}_1)$, there is a configuration of type $\mathbf{D}_6$ orthogonal to the simple fibre $G_1$ with $2$ components, which by \Cref{prop:embedE8} is contained in a second reducible fibre $G_1'\in |G_1|$ of type $I_2^*$, $I_3^*$ or $III^*$. The $I_3^*$ fibre is impossible since this would yield a section to the existing simple $III^*$ fibre (the union of $\mathbf{E}_7$ and one $\mathbf{A}_1$). The $III^*$ fibre is impossible, since otherwise the bisection of $G_1$ would intersect $G_1'$ in a component of multiplicity $3$. Hence we obtain the following graph, for which no vertex is orthogonal to all elliptic configurations:
$$
    
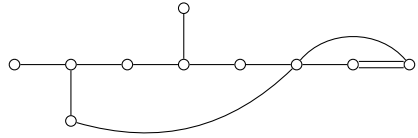
\begin{tikzpicture}[scale=0.75]
    \node (R0) at (-1,0) [nodal,fill=white] {};
\node (R1) at (0,0) [nodal,fill=white] {};
\node (R2) at (1,0) [nodal,fill=white] {};
\node (R3) at (2,0) [nodal,fill=white] {};
\node (R4) at (3,0) [nodal,fill=white] {};
\node (R5) at (4,0) [nodal,fill=white] {};
\node (R6) at (5,0) [nodal,fill=white] {};
\node (R7) at (6,0) [nodal,fill=white] {};
\node (R8) at (2,1) [nodal,fill=white] {};
\node (R9) at (0,-1) [nodal,fill=white] {};

\draw (R0)--(R1)--(R2)--(R3)--(R4)--(R5)--(R6) (R3)--(R8) (R5) to[bend left=50] (R7) (R5) to[bend left] (R9) (R1)--(R9);
\draw[double] (R6)--(R7);
    \end{tikzpicture}
$$

In the case $(\mathbf{D}_m,\mathbf{D}_n,\mathbf{A}_1)$, with $m=n=5$, consider a simple fibre $G_1$ of type $I_1^*$. There is a configuration of type $\mathbf{A}_3$ orthogonal to $G_1$, so by \Cref{prop:embedE8} the pencil $|G_1|$ admits a second reducible fibre of type $I_4$. If it is a half-fibre, we obtain by \cite[Theorem~4.1]{Martin} that $S$ is a surface with finite automorphism group of Kond\=o's type $\mathrm{II}$, and we conclude by \Cref{ex:kummer}. If it is a simple fibre, we obtain the following graph, where the bold vertices are the only ones orthogonal to all elliptic configurations:

$$
    
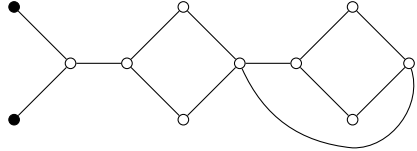
\begin{tikzpicture}[scale=0.75]
\node (R1) at (0,1) [nodal] {};
\node (R2) at (0,-1) [nodal] {};
\node (R3) at (1,0) [nodal,fill=white] {};
\node (R4) at (2,0) [nodal,fill=white] {};
\node (R5) at (3,1) [nodal,fill=white] {};
\node (R6) at (3,-1) [nodal,fill=white] {};
\node (R7) at (4,0) [nodal,fill=white] {};
\node (R8) at (5,0) [nodal,fill=white] {};
\node (R9) at (6,1) [nodal,fill=white] {};
\node (R10) at (6,-1) [nodal,fill=white] {};
\node (R11) at (7,0) [nodal,fill=white] {};

\draw (R1)--(R3) (R2)--(R3)--(R4)--(R5)--(R7) (R4)--(R6)--(R7)--(R8)--(R9)--(R11) (R8)--(R10)--(R11) (R7) to[bend right] (6,-1.5) to[bend right=55] (R11);

    \end{tikzpicture}
$$

Finally consider the case $(\mathbf{D}_m,\mathbf{A}_1,\mathbf{A}_1)$, with $6\le m \le 8$, and the simple fibre $G_1$ with two components. There is a configuration of type $\mathbf{D}_{m-2}\oplus \mathbf{A}_1$ orthogonal to $G_1$.

$$
    
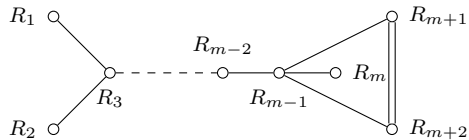
\begin{tikzpicture}[scale=0.75]
        \node (R1) at (-3, 1) [nodal, label=left:\tiny$R_1$, fill=white] {};
        \node (R2) at (-2, 0) [nodal, label=below:\tiny$R_3$, fill=white] {};
        \node (R3) at (-3,-1) [nodal, label=left:\tiny$R_2$, fill=white] {};
        \node (R4) at (0,0) [nodal, label=above:\tiny$R_{m-2}$, fill=white] {};
        \node (R5) at ( 1, 0) [nodal, label=below:\tiny$R_{m-1}$, fill=white] {};
        \node (R6) at ( 2, 0) [nodal, label=right:\tiny$R_m$, fill=white] {};
        \node (R7) at ( 3, 1) [nodal, label=right:\tiny$R_{m+1}$, fill=white] {};
        \node (R8) at ( 3,-1) [nodal, label=right:\tiny$R_{m+2}$, fill=white] {};
        \draw (R1)--(R2)--(R3) (R4)--(R5)--(R6) (R5)--(R7) (R5)--(R8);
        \draw[dashed] (R2)--(R4);
        \draw [double] (R7)--(R8);
    \end{tikzpicture}
$$

Assume first that there is a unique fibre $G_1'\in |G_1|$ containing this configuration. $G_1'$ contains the vertex $R_3$ of valency $3$, so it must be additive. Moreover the two components $R_{m-2}$ and $R_m$ of $G_1'$, that meet the bisection $R_{m-1}$ of $G_1$ and $G_1'$, must be simple components (i.e.\ of multiplicity $1$) of $G_1'$, so $G_1'$ cannot be of type $II^*$. If $G_1'$ were of type $I_l^*$ for some $l\ge 0$, then $R_1$, $R_2$, $R_{m-2}$ and $R_m$ would be the four simple components of $G_1'$. But since $R_{m-3}$ intersects $R_{m-2}$, it would also intersect $R_m$, which is impossible. Therefore the configuration of type $\mathbf{D}_{m-2}\oplus \mathbf{A}_1$ must be contained in a fibre of type $III^*$. In particular $m=8$, since $R_{m-2}$, which meets with the bisection $R_{m-1}$, must have multiplicity $1$ in $G_1'$:

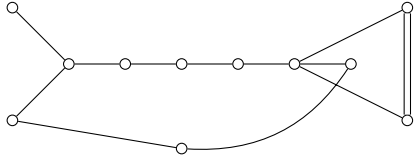
$$
    \begin{tikzpicture}[scale=0.75]
        \node (R1) at (-3, 1) [nodal,fill=white] {};
        \node (R2) at (-2, 0) [nodal,fill=white] {};
        \node (R3) at (-3,-1) [nodal,fill=white] {};
        \node (R4) at (-1,0) [nodal,fill=white] {};
        \node (R5) at (0,0) [nodal,fill=white] {};
        \node (R6) at ( 1, 0) [nodal,fill=white] {};
        \node (R7) at ( 2, 0) [nodal,fill=white] {};
        \node (R8) at (3,0) [nodal,fill=white] {};
        \node (R9) at ( 4, 1) [nodal,fill=white] {};
        \node (R10) at ( 4,-1) [nodal,fill=white] {};
        \node (R11) at (0,-1.5) [nodal,fill=white] {};
        \draw (R1)--(R2)--(R3) (R2)--(R4)--(R5)--(R6)--(R7)--(R8) (R9)--(R7)--(R10) (R3)--(R11) to[bend right] (R8);
        \draw [double] (R9)--(R10);
    \end{tikzpicture}
$$

It is straightforward to check that this graph has no vertex orthogonal to all elliptic configurations.

If instead the configuration belongs to two different fibres, the fibre $G_1'\in |G_1|$ containing the $\mathbf{D}_{m-2}$ configuration must be of type $I_{m-6}^*$, since if it were of any other type the graph would contain an additive elliptic fibre with a section, contradicting the fact that additive fibres are simple. Hence we get the following graph, where the bold vertices are the only ones orthogonal to all elliptic configurations:

$$
    
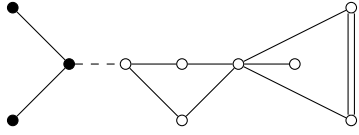
\begin{tikzpicture}[scale=0.75]
        \node (R1) at (-3, 1) [nodal] {};
        \node (R2) at (-2, 0) [nodal] {};
        \node (R3) at (-3,-1) [nodal] {};
        \node (R) at (-1,0) [nodal,fill=white] {};
        \node (R4) at (0,0) [nodal,fill=white] {};
        \node (R5) at ( 1, 0) [nodal,fill=white] {};
        \node (R6) at ( 2, 0) [nodal,fill=white] {};
        \node (R7) at ( 3, 1) [nodal,fill=white] {};
        \node (R8) at ( 3,-1) [nodal,fill=white] {};
        \node (R9) at (0,-1) [nodal,fill=white] {};
        \draw (R1)--(R2)--(R3) (R)--(R4)--(R5)--(R6) (R5)--(R7) (R5)--(R8) (R)--(R9)--(R5);
        \draw[dashed] (R2)--(R);
        \draw [double] (R7)--(R8);
    \end{tikzpicture}
$$

In all the cases we see that there are no configurations of $(-2)$-curves as in the statement that are orthogonal to all the elliptic fibres, thus proving the desired claim.

\end{proof}

\begin{proof}[Proof of \Cref{thm:overexceptional_Enriques}]
Recall that by \Cref{lemma:overexceptionalE6} there is an embedding $E'(S)\hookrightarrow \mathbf{E}_6$.
\begin{enumerate}
    \item Assume by contradiction that $\rk{E'(S)}=6$. Then $\mathbf{E}_6$ is an overlattice of the root lattice $E'(S)$. By Equation (\ref{eq:over}) the discriminant of $E'(S)$ is $3$ up to squares. 
    Since the only ADE lattices with discriminant divisible by $3$ are $\mathbf{A}_2$ and $\mathbf{E}_6$, either $E'(S)$ is isometric to $\mathbf{E}_6$ or it is of the form $\mathbf{A}_2\oplus L$ for a certain root lattice $L$ of rank $4$.
    
    The discriminant of $L$ is a square $m^2$, and by assumption $\mathbf{E}_6$ is an overlattice of $\mathbf{A}_2\oplus L$ of index $m$. In other words, there exists an isotropic subgroup of $A_{\mathbf{A}_2\oplus L}=A_{\mathbf{A}_2}\oplus A_L$ of order $m$.
    If $m$ is not divisible by $3$, then the subgroup is contained in $A_L$, and therefore $L$ admits an overlattice of discriminant $1$. This is impossible, as there are no unimodular lattices of rank $4$.
    
    Hence $L$ contains a copy of $\mathbf{A}_2$, and since the discriminant of $E'(S)$ is $3$ up to squares, we deduce that $E'(S)$ is isometric to either $\mathbf{E}_6$ or $\mathbf{A}_2^{\oplus 3}$.
    However both possibilities contradict \Cref{prop:small_overexceptional}.
    \item If the embedding $E'(S)\hookrightarrow \Num(S)$ is not primitive, $E'(S)$ admits an overlattice of index $p$ prime contained in its saturation. By \Cref{prop:conc}, $E'(S)$ contains a sublattice $\widetilde R$ isometric to $\mathbf{A}_{p-1}^{\oplus r}$, admitting an overlattice. 
    The only lattices of the form $\mathbf{A}_{p-1}^{\oplus r}$ of rank $\le 5$ and admitting an overlattice are $\mathbf{A}_1^{\oplus 4}$ and $\mathbf{A}_1^{\oplus 5}$ by \Cref{lem:overnonroot}. 
    Since $\rk{\widetilde R}\le 5$, we have that $\widetilde R$ is isometric to $\mathbf{A}_1^{\oplus 4}$ or $\mathbf{A}_1^{\oplus 5}$. 
    In both cases $E'(S)$ contains four disjoint curves, contradicting \Cref{prop:small_overexceptional}.
    \item Let $\Lambda$ be the sublattice of $\Pic(X)$ spanned by the components of the pullbacks of the $(-2)$-curves in $E'(S)$. The $(-2)$-curves in $E'(S)$ are vertical components of any elliptic fibration on $S$, and $E'(S)$ contains no complete fibre by \Cref{lemma:overexceptionalE6}, hence the union $\Delta$ of the $(-2)$-curves in $E'(S)$ is simply connected. In particular $\Delta$ splits under the K3 cover, and $\Lambda\cong E'(S)^{\oplus 2}$ is isometric to the direct sum of two copies of $E'(S)$.
    
    If the embedding $\Lambda \hookrightarrow \Pic(X)$ is not primitive, then by \Cref{prop:primitive}, $\Lambda$ admits an overlattice of index $p$ prime that is not a root lattice. By \Cref{prop:conc} this overlattice is concentrated at a sublattice $\widetilde R \subseteq \Lambda$ spanned by $(-2)$-curves and isometric to $\mathbf{A}_{p-1}^{\oplus r}$. Since $\widetilde R$ admits an overlattice that is not a root lattice, and it has rank $\le 10$ by part (1), by \Cref{lem:overnonroot} $\widetilde R$ is isometric to $\mathbf{A}_1^{\oplus r}$ for a certain $8\le r \le 10$. But then the images of these $r\ge 8$ disjoint $(-2)$-curves on the Enriques surface $S$ are at least $\lceil r/2 \rceil \ge 4$ disjoint $(-2)$-curves in $E'(S)$, contradicting \Cref{prop:small_overexceptional}.
\end{enumerate}
\end{proof}

\begin{corollary} \label{cor:overexceptional_K3}
Let $X$ be a K3 surface over an algebraically closed field of characteristic different from $2$ covering an Enriques surface. The over-exceptional lattice $E'(X)$ of $X$ has rank $\le 10$ and it embeds primitively into $\Pic(X)$.
\end{corollary}
\begin{proof}
Let $S$ be an Enriques surface covered by $X$. If a $(-2)$-curve $C$ on $X$ belongs to the over-exceptional lattice $E'(X)$, it is orthogonal to all elliptic fibres on $X$. In particular it is orthogonal to all the elliptic fibres on $X$ that are pullbacks of elliptic fibres on the Enriques surface $S$. Therefore $E'(X)\subseteq E'(S)^{\oplus 2}$. Moreover the embedding $E'(X)\hookrightarrow E'(S)^{\oplus 2}$ is primitive, since a $\ZZ$-basis for $E'(X)$ is given by a sub-$\ZZ$-basis of $E'(S)^{\oplus 2}$. Hence we conclude by \Cref{thm:overexceptional_Enriques}.
\end{proof}

\begin{remark} \label{rk:finitelist}
 \begin{enumerate}
     \item The conditions found in \Cref{cor:overexceptional_K3} for the over-exceptional lattice $E'(X)$ do not hold for every K3 surface $X$. As an example, let $X$ be a K3 surface with $\Pic(X)=\mathbf{U}\oplus \mathbf{E}_8^{\oplus 2}$. By \cite[\S4.2]{KondoK3} $X$ has finite automorphism group, and the dual graph of its $(-2)$-curves is as follows:
     $$
    
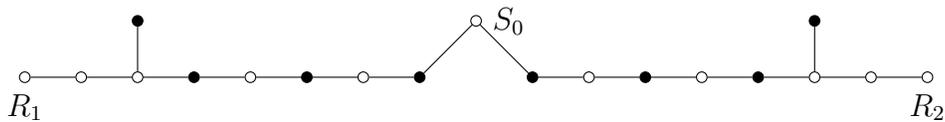
\begin{tikzpicture}[scale=0.75]
        
        \node (R1) at (0, 0) [nodal, fill=white, label=below:$R_1$] {};
        \node (R2) at (1, 0) [nodal, fill=white] {};
        \node (R3) at (2,0) [nodal, fill=white] {};
        \node (R4) at (3,0) [nodal] {};
        \node (R5) at (4,0) [nodal, fill=white] {};
        \node (R6) at ( 5, 0) [nodal] {};
        \node (R7) at ( 6, 0) [nodal, fill=white] {};
        \node (R8) at ( 7, 0) [nodal] {};
        \node (R9) at ( 2,1) [nodal] {};
        \node (S) at (8,1) [nodal, fill=white, label=right:$S_0$] {};
        \node (R11) at (9, 0) [nodal] {};
        \node (R12) at (10, 0) [nodal, fill=white] {};
        \node (R13) at (11,0) [nodal] {};
        \node (R14) at (12,0) [nodal, fill=white] {};
        \node (R15) at (13,0) [nodal] {};
        \node (R16) at ( 14, 0) [nodal, fill=white] {};
        \node (R17) at ( 15, 0) [nodal, fill=white] {};
        \node (R18) at ( 16, 0) [nodal, fill=white, label=below:$R_2$] {};
        \node (R19) at ( 14,1) [nodal] {};

        \draw (R1)--(R2)--(R3)--(R4)--(R5)--(R6)--(R7)--(R8)--(S)--(R11)--(R12)--(R13)--(R14)--(R15)--(R16)--(R17)--(R18) (R3)--(R9) (R16)--(R19);
    \end{tikzpicture}
$$

There are precisely two elliptic fibrations on $X$: the first one has two reducible fibres of type $II^*$ and $S_0$ is its unique section, while the second one has a reducible fibre of type $I_{12}^*$ and $R_1,R_2$ are its sections. We obtain that $E'(X)$ is spanned by the other (unlabelled) $16$ curves, and it is isometric to $\mathbf{D}_8^{\oplus 2}$. Moreover the embedding $E'(X)\hookrightarrow \Pic(X)$ is not primitive, as the sum $D$ of the eight bold vertices is divisible by $2$ in $\Pic(X)$. Indeed $D$ intersects every element in $\Pic(X)$ with even multiplicity, hence $\frac{1}{2}D$ belongs to $\Pic(X)^\vee = \Pic(X)$ (by unimodularity of $\Pic(X)$). The saturation of $E'(X)$ in $\Pic(X)$ is not a root lattice, since it is concentrated at a sublattice isometric to $\mathbf{A}_1^{\oplus 8}$. It follows by \Cref{rk:geometric} that the fundamental group of the complement $X\setminus Z$, where $Z$ is the union of the $(-2)$-curves in the over-exceptional lattice, is isomorphic to $\ZZ/2$.
     
\item Nikulin proves \cite[Theorem~4.4]{Nikulin} that there exists a finite list $\CS'$ of Picard lattices of K3 surfaces of rank $\ge 6$ with a non-zero exceptional lattice. We deduce that there exists a finite list $\CS\subseteq \CS'$ of Picard lattices of K3 surfaces of rank $\ge 6$ with a non-zero over-exceptional lattice.

However the proof of finiteness of $\CS'$ is non-constructive. The K3 surfaces $X$ with non-trivial $E(X)$ can be grouped depending on whether $E(X)$ is \emph{hyperbolic}, \emph{parabolic} (i.e.\ negative semidefinite
with one-dimensional kernel) or \emph{elliptic} (i.e.\ negative definite). Of these, the first two classes are fully known. A K3 surface $X$ admits a hyperbolic exceptional lattice $E(X)$ if and only if $E(X)=\Pic(X)$ if and only if it has a finite automorphism group by \cite[\S4]{Nikulin}, and these surfaces have been classified by Nikulin \cite{Nikulin_2} and Vinberg \cite{Vinberg} (see \cite[\S15.2, p.\ 317]{Huybrechts} for a comprehensive list of references).

On the other hand, K3 surfaces with a parabolic $E(X)$ admit a unique elliptic fibration with infinite stabiliser, and from the point of view of complex dynamics they form an important class of K3 surfaces, usually called \emph{of zero entropy} (see \cite{Cantat}). A classification of the Picard lattices of such surfaces was obtained by the second author in \cite{Mezzedimi2} and independently by \cite{XY}.

Lastly, K3 surfaces with a negative definite $E(X)$ admit infinitely many elliptic fibrations with infinite stabiliser by \cite[Theorem~5.1]{Nikulin}, and from the point of view of complex dynamics they exhibit an automorphism of \emph{positive entropy}. There are examples of such K3 surfaces with a non-zero exceptional lattice, such as the Kummer surfaces in Example \ref{ex:kummer}. 
No classification of the K3 surfaces with a non-zero, negative definite exceptional lattice is known at the moment.
 \end{enumerate}
\end{remark}

\section{Hilbert Property and elliptic fibrations}\label{sec:demeio}
The following theorem is a generalisation of a result by Demeio \cite[Theorem~1.1]{Demeio}, with the original method going back to Corvaja and Zannier \cite[Theorem~1.6]{CorvajaZannier}. 
\begin{theorem}\label{thm:demeio}
  Let $X$ be a smooth, geometrically integral, projective surface over a finitely generated field $K$ of characteristic $0$ with pairwise distinct elliptic fibrations \[\pi_i:X\to\BP^1_K,\, i=1,\dots,n\quad (n\geq 2).\] Let $Z\subset X_K$ be the union of all prime divisors on $X_K$ which are, for each $i=1,\dots,n$, contained in a fibre of $\pi_i$.
  
  If $\piet((X\setminus Z)_{\ov K})\cong\piet(X_{\ov K})$ and $X(K)$ is Zariski-dense in $X$, then $X$ has the weak Hilbert Property.
\end{theorem}
\begin{proof}
 The first step of the proof is as in \cite[Theorem~1.1]{Demeio}. Concluding verbatim, one only needs to consider non-\'etale covers $\phi_i:Y_i\to X$ that are \'etale outside $Z$. (While the argument is stated for number fields, Faltings's Theorem applies to finitely generated fields of characteristic $0$. The proof of \cite[Theorem~3.2]{Demeio} relies on Merel's uniform torsion bound, which also generalises \cite[Footnote~1]{CadoretTamagawa} -- although the bound may not be explicit in this case.) It remains to be shown that every finite \'etale cover of $X\setminus Z$ can be extended to $X$.
 
  Fix an embedding $K\hookrightarrow \BC$. Recall that for a geometrically connected variety $\piet(X_{\ov K})\simeq\piet(X_\BC)$ \cite[Proposition~XIII.4.6]{SGA1}, and $\piet(X_\BC)$ is the profinite completion of the topological fundamental group $\pitop(X_\BC)$.
  
  Let $Z_1,\dots,Z_k$ be the connected components of $Z$. By the Seifert--van Kampen theorem, $\pitop(X_\BC)$ is the amalgamated product of $\pitop((X\setminus Z_1)_\BC)$ and $\pitop(Z_{1,\BC})$. Hence, $\pitop((X\setminus Z_1)_\BC)\to\pitop(X_\BC)$ is surjective. Continuing by induction, the inclusion $i:X\setminus Z\hookrightarrow X$ induces an epimorphism $\piet((X\setminus Z)_{\ov K})\twoheadrightarrow\piet(X_{\ov K})$. Because $\piet(X_{\ov K})$ as a topologically finitely generated profinite group is Hopfian \cite[Proposition~2.3]{Ribes}, it follows that this is an isomorphism of fundamental groups and we obtain the following commutative diagrams with exact rows \cite[Proposition~5.6.1]{Szamuely}:
 \[\begin{tikzcd}
    0 \arrow[r] & \piet((X\setminus Z)_{\ov K}) \arrow[r] \arrow[d,"\simeq"] & \piet(X\setminus Z) \arrow[d,"i_*"]\arrow[r]  & \Gal(\ov K/K) \arrow[r] \arrow[d,equal] & 0\\
    0 \arrow[r] & \piet(X_{\ov K}) \arrow[r] & \piet(X) \arrow[r] & \Gal(\ov K/K) \arrow[r] & 0.
   \end{tikzcd}
\]
 By the five lemma, the middle vertical arrow is an isomorphism. Thus, the categories of finite \'etale covers over $X$ and $X\setminus Z$ are equivalent.
\end{proof}

Note that, if $X$ is a minimal algebraic surface of Kodaira dimension $0$ as in \Cref{thm:demeio} over an algebraically closed field, the divisor $Z$ coincides with the union of the $(-2)$-curves in the over-exceptional lattice $E'(X)$, since any reducible fibre of a relatively minimal elliptic fibration is the union of smooth rational curves.

If one restricts to minimal surfaces, the applicability of \Cref{thm:demeio} is limited by the following generalisation of a well-known result which for example appears in \cite[Lemma~12.18]{SchuettShioda}:
\begin{proposition}
A minimal algebraic surface over an algebraically closed field with two distinct elliptic fibrations has Kodaira dimension $0$.
\end{proposition}
\begin{proof}
 Let $X$ be such a surface with two distinct elliptic fibrations $\pi_1$, $\pi_2$. Let $F_1$, $F_2$ be smooth, non-multiple fibres of $\pi_1$, $\pi_2$ passing through a point $P\in X$. For $i=1,2$, $K_X$ is linearly equivalent to a sum of fibres of $\pi_i$ by Kodaira's canonical bundle formula \cite[Theorem~V.12.1]{BarthHulek} for relatively minimal elliptic fibrations. Since all fibres of an elliptic fibration become numerically equivalent after taking sufficiently high multiples, there exist integers $a,b_1,b_2\geq0$ such that $aK_X=b_1F_1=b_2F_2$ in $\Num(X)$. If $b_1$ or $b_2$ is $0$, then $K_X$ is torsion. Otherwise, $F_1F_2=\frac{a^2}{b_1b_2}K_X^2=0$, but this is impossible since the curves $F_1$ and $F_2$ intersect in $P$.
\end{proof}

\begin{remark}
Since the weak Hilbert Property is a birational invariant of smooth, proper varieties (see \cite[Proposition~3.1]{CDJLZ}), it would be interesting to see whether it is possible to apply \Cref{thm:demeio} after purposefully blowing up a minimal surface.
\end{remark}

We now complete the proofs of Theorems \ref{thm:enriques} and \ref{thm:K3}.

\begin{proof}[Proof of Theorem A]
Up to passing to a finite field extension of $K$, we may assume that $\Pic(S_K)=\Pic(S_{\ov K})$ and by \cite{BogomolovTschinkel} that $S(K)$ is Zariski-dense in $S$.

Let $C_1,\ldots,C_n$ be the curves in the over-exceptional lattice $E'(S)$, and denote by $Z$ their union. By \Cref{thm:demeio} it is sufficient to show that $\pitop((S\setminus Z)_{\BC})\simeq\ZZ/2$ and thus $\piet((S\setminus Z)_{\ov K})\simeq\ZZ/2$ (the étale fundamental group being the profinite completion of the topological fundamental group as in the proof of \Cref{thm:demeio}).

Let $X$ be the K3 cover of $S$. Since the curves $C_i$ are smooth and rational, the pullback of $C_i$ on $X$ consists of two disjoint $(-2)$-curves $C_i'$ and $C_i''$. Denote by $Z'\subseteq X$ the pullback of $Z$.
Since the restriction $X\setminus Z'\to S\setminus Z$ is an \'etale double cover, it is sufficient to show that $\pitop((X\setminus Z')_{\ov K})=\{1\}$.

This follows from \cite[Theorem~4.3]{ShimadaZhang}. Indeed the curves in $Z'$ span the root lattice $E'(S)^{\oplus 2}\hookrightarrow \Pic(X)$, and we have shown in \Cref{thm:overexceptional_Enriques} that this embedding is primitive, and that $\rk{E'(S)^{\oplus 2}}\le 10$. 
\end{proof}

\begin{proof}[Proof of Theorem B]
Up to passing to a finite field extension of $K$, we may assume that $\Pic(X_K)=\Pic(X_{\ov K})$ and, since $X$ admits an elliptic fibration, that $X(K)$ is Zariski-dense in $X$ by \cite{BogomolovTschinkel_2}.

Let $\CS$ be the finite list of lattices with non-zero over-exceptional lattice as in \Cref{rk:finitelist}(2).
We claim that any K3 surface $X$ with Picard rank $\ge 6$ and $\Pic(X)$ not in $\CS$ admits at least two distinct elliptic fibrations. Indeed every K3 surface of Picard rank $\ge 6$ is elliptic by \cite[Proposition~11.1.3(ii)]{Huybrechts}, and if $X$ admits only one elliptic fibration $|F|$, then by definition $E'(X)$ contains the elliptic curve $F$, hence $E'(X)\ne 0$ and in particular $\Pic(X)$ belongs to the finite list $\CS$.

Let $C_1,\ldots,C_n$ be the $(-2)$-curves in the over-exceptional lattice $E'(X)$, and denote by $Z$ their union. By \Cref{thm:demeio} it is sufficient to show that $X\setminus Z$ is simply connected.

If $X$ satisfies assumption (2), then the simply connectedness is obvious since $E'(X)=0$ by assumption, and hence $Z=\varnothing$. Therefore assume that $X$ satisfies assumption (1), i.e.\ that it admits two distinct elliptic fibrations $|F_1|$ and $|F_2|$, and that $\rho(X)<10$.
By definition the curves in the over-exceptional lattice are orthogonal to $F_1$ and $F_2$, so we have an embedding $E'(X)\hookrightarrow \langle F_1,F_2\rangle^\perp$. By assumption this implies that the rank of $E'(X)$ is strictly less than $8$.
We claim that the embedding $E'(X)\hookrightarrow \Pic(X)$ is primitive. If this were not the case, by \Cref{prop:primitive} the root lattice $E'(X)$ would admit an overlattice that is not a root lattice, contradicting \Cref{prop:overlattice_ge8}.
Therefore the simply connectedness of the complement $X\setminus Z$ follows from \cite[Theorem~4.3]{ShimadaZhang}.
\end{proof}

The following theorem shows that the Picard lattices of K3 surfaces with an Enriques involution are not contained in the list $\CS$.
\begin{theorem}\label{thm:enriques_cover}
 Every K3 surface $X$ over a finitely generated field $K$ of characteristic $0$ with an Enriques involution has the potential Hilbert property.
 
 More precisely, if $X(K)$ is Zariski-dense in $X$ and $\Pic(X)=\Pic(X_{\ov K})$, then $X$ has the Hilbert Property.
\end{theorem}
\begin{proof}
Let $S$ be any Enriques surface covered by $X$. Up to passing to a finite field extension of $K$, we may assume that $\Pic(X_K)=\Pic(X_{\ov K})$ and, since $S$ (and thus $X$) admits an elliptic fibration, that $X(K)$ is Zariski-dense in $X$ by \cite{BogomolovTschinkel_2}.

Denote by $Z$ the union of the $(-2)$-curves in the over-exceptional lattice $E'(X)$ of $X$. By \Cref{thm:demeio} it is sufficient to show that $X\setminus Z$ is simply connected. This follows from \cite[Theorem~4.3]{ShimadaZhang}, since by \Cref{cor:overexceptional_K3} the embedding $E'(X)\hookrightarrow \Pic(X)$ is primitive and $\rk{E'(X)}\le 10$.
\end{proof}

\begin{remark}\label{rem:uniform}
If $V$ is an Enriques surface or a K3 surface as in Theorem B, with an elliptic fibration over $K$, then the field extension $K'/K$ over which $V$ acquires the Hilbert Property can be uniformly bounded. In fact, it suffices to choose $K'$ such that $V(K')$ is Zariski-dense in $V$ and $\Num(V_{K'})=\Num(V_{\ov K})$.

Let $c_{n}$ be the largest possible cardinality of a finite subgroup of $\GL(n,\BZ)$ and set $C=4608c_{20}$. We use \cite[Lemma~4.4]{LaiNakahara} and \cite[Theorem~4.14]{LaiNakahara}, which are formulated for $K$ a number field, but all the proof steps work for finitely generated field of characteristic $0$. (A version of Merel's uniform torsion bound for higher-dimensional fields is explained in \cite[Footnote~1]{CadoretTamagawa}. Like in the number field case, this torsion bound depends on $K$ but its value does not affect $C$.)

It follows that $K'$ can be chosen such that $[K':K]\leq C$. In the situation of Theorem~A, where $V=S$ is an Enriques surface, the methods of \cite{LaiNakahara} and \cite{BogomolovTschinkel} can be used to improve this bound and replace $c_{20}$ by $c_{10}$. For explicit values of $c_n$, the reader may consult \cite{cn}.
\end{remark}

\printbibliography

\end{document}